\documentclass[final,leqno]{article}
\usepackage[section]{placeins}
\usepackage[utf8]{inputenc}
\usepackage[english]{babel}
\usepackage{amsmath,amsfonts,amssymb,amsthm}
\usepackage{paralist}
\usepackage{graphics} 
\usepackage{epsfig} 
\usepackage{graphicx}
\usepackage{subfig}
\usepackage{caption}
\usepackage{epstopdf}
\usepackage[colorlinks=true]{hyperref}
 
\usepackage{float} 
\usepackage{bm}
\usepackage{fancyhdr}
\usepackage{url}
\usepackage{cleveref}

\hypersetup{urlcolor=blue, citecolor=red}

\newtheorem{theorem}{Theorem}[section]

\newtheorem*{proof*}{Proof}

\newtheorem{remark}{Remark}

\begin{document}
\title{A Partial Differential Equation Model with Age-Structure and Nonlinear Recidivism: Conditions for a Backward Bifurcation and a General Numerical Implementation}
\maketitle

\begin{center}
\author{Fabio Sanchez\footnote{Corresponding author: \url{fabio.sanchez@ucr.ac.cr}\\
Present address: Centro de Investigaci\'on en Matem\'atica Pura y Aplicada (CIMPA), Escuela de Matem\'atica, Universidad de Costa Rica. San Pedro de Montes de Oca, San Jos\'e, Costa Rica, 11501.},
Juan G. Calvo\footnote{Centro de Investigaci\'on en Matem\'atica Pura y Aplicada (CIMPA), Escuela de Matem\'atica, Universidad de Costa Rica. San Pedro de Montes de Oca, San Jos\'e, Costa Rica, 11501. Email: \url{juan.calvo@ucr.ac.cr}},
Esteban Segura\footnote{Centro de Investigaci\'on en Matem\'atica Pura y Aplicada (CIMPA), Escuela de Matem\'atica, Universidad de Costa Rica. San Pedro de Montes de Oca, San Jos\'e, Costa Rica, 11501. Email: \url{esteban.seguraugalde@ucr.ac.cr}} and 
Zhilan Feng\footnote{Department of Mathematics, Purdue University, West Lafayette, IN, USA, 47907. Email: \url{zfeng@math.purdue.edu}}}
\end{center}

\begin{abstract}
\noindent We formulate an age-structured three-staged nonlinear partial differential equation model that features {\it nonlinear} recidivism to the infected ({\it infectious}) class from the {\it temporarily} recovered class. Equilibria are computed, as well as local and global stability of the {\it infection-free} equilibrium. As a result, a backward-bifurcation exists under necessary and sufficient conditions. A generalized numerical framework is established and numerical experiments are explored for two positive solutions to exist in the {\it infectious} class.
\end{abstract}


\section{Introduction} \label{sec:intro}
Ordinary differential equation models with nonlinear recidivism have been explored in previous studies~\cite{arino2003global,brauer2004,brauer2011,ccc2004,feng2000model,sanchez2007,song2006,song2013}. These models exhibit a phenomena known as a {\it backward bifurcation}, a behavior that is strongly correlated to initial conditions, and more specifically, to the initial number of infectious individuals. The implications of a backward bifurcation in epidemic models are crucial in understanding how to develop control mechanisms for the disease. 

The main concept in these models is the study of the threshold quantity $\mathcal{R}_0$, {\it the basic reproductive number}, which represents the number of secondary infections caused by a infectious individual when introduced into a mostly susceptible population. Typically, the condition of having $\mathcal{R}_0<1$ is sufficient to have the disease die out in compartmental models. In other words, the infection-free equilibrium is stable. Conversely, the disease prevails and the endemic equilibrium becomes asymptotically stable when $\mathcal{R}_0>1$. However, if there is a backward bifurcation, then $\mathcal{R}_0<1$ is not a sufficient condition for the disease to die out. Furthermore, efforts to control the disease will depend on the density of infectious individuals in the population or having $\mathcal{R}_0 \ll 1$. 

In ordinary differential equation models, when studying the bifurcation behavior, it is represented by the graph of the infected class versus $\mathcal{R}_0$. This is represented as a function of the transmission rate. This behavior translates into the homologous nonlinear partial differential equation model. The bifurcation in our model is described graphically by a surface plot represented by the infected class at a steady state distribution as a function of age and $\mathcal{R}_0$. The main difference is the age dependence which can lead to the development of better control efforts using the individuals age distribution.

Other studies have included age-structure in their models~\cite{allen2008mathematical,holmes1994partial,kribs2002,martcheva2003progression,shim2006,sanchez2018x}. However, in the model presented here, besides including age-structure, we also include the possibility of recidivism, which leads to a backward bifurcation under some conditions. A strictly theoretical approach was studied in \cite{ccc2002}. 

The rest of this paper is organized as follows. In \Cref{sec:model}, we describe the homologous partial differential equation model, based on the nonlinear ordinary differential equation model presented in~\cite{sanchez2007}. In \Cref{sec:freesteadyState}, we study the infection-free non-uniform steady state distribution, its stability, and we define the {\it basic reproductive number}, $\mathcal{R}_0$. We study the endemic non-uniform steady state distribution in \Cref{sec:steadyState}, where we also analyze the existence of two positive solutions if necessary and sufficient conditions hold, which results in a backward bifurcation. In \Cref{sec:exp}, a general numerical framework is presented, including an example based on the parameters used in \cite{sanchez2007}, as well as an example with age-dependent parameters. Finally, in \Cref{sec:conc} we give some conclusions of the model.


\section{A model with nonlinear recidivism} \label{sec:model}
We construct an age-dependent nonlinear partial differential equation model where $S(t,a)$, $I(t,a)$ and $R(t,a)$ represent susceptible, infected and temporarily recovered individuals at time $t$ and age $a$, respectively. Susceptible individuals can become infected by having contact with an infectious individual at rate $\beta(a) B(t,a)$, where $B(t,a)$ is the probability that a person contacted by a susceptible of age $a$ is infectious at time $t$ and $\beta(a)$ is the transmission rate at age $a$; they can also exit the system at rate $\mu(a)$, where $\mu(a)$ is the exit rate of the system at age $a$. Infected individuals receive treatment at rate $\phi(a)$, can recover at rate $\gamma(a)$ or exit the system at rate $\mu(a)$. Analogously, temporarily recovered individuals can become infected ({\it subsequent infections}) by having contact with an infectious individual at rate $\rho(a) B(t,a)$ and can exit the system at rate $\mu(a)$. In this case, the class $R(t,a)$ acts as another \lq\lq susceptible" class.

The age-dependent mixing contact structure is modeled via the mixing density $p(t,a,a')$, which gives the proportion of contact that individuals of age $a$ have with individuals of age $a'$, given that they had contact with somebody at time $t$. We will restrict ourselves to the case of proportional mixing; this is, $$p(t,a,a') \equiv p(t,a) =  \frac{c(a)n(t,a)}{\int_0^\infty c(a) n(t,a)\ da}, $$ with $c(a)$ the age-specific per capita contact/activity rate; see, e.g., \cite{shim2006}. We then define the force of infection $$ B(t) := \int_0^\infty \frac{I(t,a)}{n(t,a)} p(t,a)da,$$ where $n(t,a):= S(t,a)+I(t,a)+R(t,a)$ is the total population.

The model we just described is given by the system of partial differential equations
\begin{align}\label{PDE}
\left( \frac{\partial}{\partial t} + \frac{\partial}{\partial a} \right) S(t,a) &=  - \beta(a) S(t,a) B(t) -\mu(a)S(t,a),\\
\left( \frac{\partial}{\partial t} + \frac{\partial}{\partial a} \right) I(t,a) &= \beta(a) S(t,a) B(t) -(\mu(a) + \phi(a) + \gamma(a)) I(t,a)+\rho(a) R(t,a) B(t),\nonumber\\
\left( \frac{\partial}{\partial t} + \frac{\partial}{\partial a} \right) R(t,a) &=  (\phi(a) + \gamma(a)) I(t,a)- \rho(a) R(t,a) B(t) - \mu(a) R(t,a),\nonumber
\end{align} 
for $t>0$, $a>0$, and boundary conditions for $t=0$ and $a=0$ given by 
\begin{align*}
S(t,0)&=\Lambda, & I(t,0)&=0, & R(t,0)&=0,\nonumber\\
S(0,a)&=S_0(a), & I(0,a)&=I_0(a), & R(0,a)&=R_0(a),
\end{align*}
where $\Lambda$ denotes the birth rate (assumed constant), and $S_0(a), I_0(a), R_0(a)$ are given initial conditions. For simplicity, we will assume that $\beta(a)$, $\mu(a)$, $\phi(a)$, $\gamma(a)$, $\rho(a)$ are continuous functions.

The total population $n(t,a)$ satisfies the initial value problem
\begin{align*}
\left( \frac{\partial}{\partial t} +\frac{\partial}{\partial a} \right) n(t,a) &= -\mu(a)n(t,a),\\
n(t,0) &= \Lambda,\\
n(0,t) &= S_0(a)+I_0(a)+R_0(a) =: n_0(a),
\end{align*}
with $t>0$ and $a>0$, where $\Lambda$ represents the constant rate at which newborns enter the population. When solving by using the method of characteristics we obtain 
\begin{equation*}
    n(t,a) = \left\lbrace
    \begin{array}{rl}
        n_0(a-t) \dfrac{\mathcal{F}(a)} {\mathcal{F}(a-t)} & {\rm if }\ t<a,  \\
        \Lambda \mathcal{F}(a) & {\rm if }\ t>a,
    \end{array}
    \right.
\end{equation*}
where $$\mathcal{F}(a)=e^{\displaystyle -\int_0^a \mu(h)dh}.$$ 

We note that $n(t,a)$ is continuous if and only if the compatibility condition $n_0(0)=\Lambda$ is satisfied. At demographic steady state, $$\lim_{t\rightarrow\infty} n(t,a)=\Lambda \mathcal{F}(a),$$ and therefore $$p_\infty(a) := \lim_{t\rightarrow\infty} p(t,a)=\frac{c(a) \mathcal{F}(a)}{\int_0^\infty c(a) \mathcal{F}(a)\ da}.$$  
We rescale variables by making the substitution
\begin{align*}
&s(t,a) := \dfrac{S(t,a)}{n(t,a)},\quad i(t,a) := \dfrac{I(t,a)}{n(t,a)},\quad r(t,a) := \dfrac{R(t,a)}{n(t,a)}.
\end{align*}
Thus, 
$$B(t) = \int_0^\infty i(t,a) p(t,a)\ da$$ 
and system \eqref{PDE} becomes
\begin{subequations} \label{pde}
\begin{align}
\begin{split} \label{pde_s}
\left( \frac{\partial}{\partial t} + \frac{\partial}{\partial a} \right) s(t,a) &= - \beta(a) s(t,a) B(t),
\end{split}\\
\begin{split} \label{pde_i}
\left( \frac{\partial}{\partial t} + \frac{\partial}{\partial a} \right) i(t,a) &= \beta(a) s(t,a) B(t) - (\phi(a)+\gamma(a)) i(t,a)+\rho(a) r(t,a) B(t),
\end{split}\\
\begin{split} \label{pde_r}
\left( \frac{\partial}{\partial t} + \frac{\partial}{\partial a} \right) r(t,a) &=  (\phi(a)+\gamma(a)) i(t,a)- \rho(a) r(t,a) B(t),
\end{split}
\end{align}
along with boundary conditions
\begin{align} \label{pde_ic}
s(t,0)&=1, & i(t,0)&=0, & r(t,0)&=0,\nonumber\\
s(0,a)&=s_0(a), & i(0,a)&=i_0(a), & r(0,a)&=r_0(a),
\end{align}
\end{subequations}
where $s_0(a)+i_0(a)+r_0(a) = 1$. 

\section{Infection-free non-uniform steady state distribution and basic reproductive number} \label{sec:freesteadyState}
In this section we explore the infection-free steady state distribution and study its local and global stability. We also compute the {\it basic reproductive number} $\mathcal{R}_0$, and the critical value $\mathcal{R}_C$ that will determine global stability to the infection-free steady state.

For the infection-free non-uniform steady state distribution, system \eqref{pde} simplifies to
\begin{subequations} \label{freepde}
\begin{align}
\begin{split} \label{freePDE_s*}
\dfrac{ds^*(a)}{da}  &= 0,
\end{split}\\
\begin{split} \label{freePDE_i*}
\dfrac{di^*(a)}{da} &= 0 ,
\end{split}\\
\begin{split} \label{freePDE_r*}
\dfrac{dr^*(a)}{da} &=  0,
\end{split}
\end{align}
with $a > 0$ and along with boundary conditions
\begin{equation} \label{PDE_ic***}
s^*(0)=1,\quad  i^*(0)=0,\quad r^*(0)=0.
\end{equation}
\end{subequations}

Thus, system \eqref{freepde} supports the infection-free non-uniform 
distribution given by
\begin{equation} \label{solfreepde}
s^*(a) = 1,\quad i^*(a) = 0,\quad r^*(a) =  0.
\end{equation}

The local stability of \eqref{solfreepde} is studied using the perturbations
\begin{align*}
s(t,a) &= \hat{s}(a)e^{\lambda t} + s^*(a) = \hat{s}(a)e^{\lambda t} + 1,\\
i(t,a) &= \hat{i}(a)e^{\lambda t},\nonumber\\
r(t,a) &= \hat{r}(a)e^{\lambda t}, \nonumber\\
B(t)  &= B_0 e^{\lambda t}, \nonumber
\end{align*}
where
\begin{equation} \label{b_o}
B_0 = \int_0^\infty \hat{i}(a)p_\infty (a)da.
\end{equation}
We then have that
\begin{equation*}
\left( \frac{\partial}{\partial t} + \frac{\partial}{\partial a} \right) i(t,a) = \lambda \hat{i}(a)e^{\lambda t} + \dfrac{d \hat{i}(a)}{da} e^{\lambda t},
\end{equation*}
and from \eqref{pde_i}, we obtain 
\begin{align*}
\lambda \hat{i}(a) + \dfrac{d \hat{i}(a)}{da} &= \beta(a) s^*(a) B_0 -(\phi(a)+\gamma(a)) \hat{i}(a) + \rho(a) r^*(a) B_0 \\
&= \beta(a) B_0 - (\phi(a)+\gamma(a)) \hat{i}(a).
\end{align*}
Thus, $\hat{i}(a)$ solves the first-order linear ordinary differential equation
\begin{align*}
\dfrac{d \hat{i}(a)}{da} + ( \lambda +  \phi(a)+\gamma(a)) \hat{i}(a) = \beta(a) B_0\quad \text{with}\quad \hat{i} (0)=0,
\end{align*}
which solution is given by
\begin{align*}
\hat{i} (a) 
&=  B_0 \int_0^a  \beta(h)\ e^{-\lambda (a-h) -\int_h^a (\phi(k) +\gamma(k) )dk}\ dh.
\end{align*} 
Multiplying by $p_\infty(a)$ and integrating, we get from \eqref{b_o} that
$$
B_0 = B_0 \int_0^\infty \int_0^a  p_\infty (a) \beta(h) e^{-\lambda (a-h) - \int_h^a (\phi(k) + \gamma(k))dk}  dh da.
$$
Therefore, for $B_0 \neq 0$,
$$
1 = \int_0^\infty \int_0^a  p_\infty (a) \beta(h) e^{-\lambda (a-h) - \int_h^a (\phi(k) +\gamma(k))dk}  dh da.
$$
Then we have the corresponding Euler-Lotka equation (see, e.g., \cite[Chapter 20]{kot2001elements})
\begin{equation} \label{Eul_Lot}
G(\lambda):= \int_0^\infty \int_0^a  p_\infty (a) \beta(h) e^{-\lambda (a-h) - \int_h^a (\phi(k)+\gamma(k))dk}  dh da =1.
\end{equation}
We then define the control adjusted reproductive number, $\mathcal{R}_0$, by
\begin{equation} 
\mathcal{R}_0 :=G(0) = \int_0^\infty \int_0^a  p_\infty (a) \beta(h) e^{- \int_h^a (\phi(k)+\gamma(k))dk}  dh da.
\end{equation}
We also consider
\begin{equation*} 
\mathcal{R}_C := \int_0^\infty \int_0^a  p_\infty (a) \beta(h)  dh da.
\end{equation*}

We then establish the following theorem:

\begin{theorem}
If $\mathcal{R}_0 < 1$, the infection-free non-uniform steady state distribution~\eqref{solfreepde} is locally asymptotically stable and unstable if $\mathcal{R}_0>1$. \label{thm:global}
\end{theorem}
\begin{proof} 
Observe that $G^\prime (\lambda)<0$ and $G^{\prime \prime}(\lambda) >0$ for all $\lambda$, and that $\displaystyle \lim_{\lambda\rightarrow \infty} G(\lambda) = 0$ and $\displaystyle \lim_{\lambda\rightarrow -\infty} G(\lambda) = \infty$. Then $G(\lambda)$ is a continuous, strictly decreasing, convex function of $\lambda$ that takes on all positive values. Thus \eqref{Eul_Lot} has exactly one real root and it is a dominant root (see \cite[Chapter 20]{kot2001elements}).
Therefore we have that \eqref{Eul_Lot} has a unique negative real solution if and only if $\mathcal{R}_0<1$ and it has a unique positive real solution if and only if $\mathcal{R}_0>1$.
Then, the fact that the unique real root is dominant guarantees the local asymptotic stability of the infection-free non-uniform steady state distribution, i.e., it is locally asymptotically stable if $\mathcal{R}_0 <1$ and unstable if $\mathcal{R}_0 >1$.
\end{proof}

We also have the following result about global stability:
\begin{theorem}\label{th:globalStab}
Assume that $\mathcal{R}_{C}  < 1.$ Then, the infection-free steady state distribution of System \eqref{pde} is globally asymptotically stable.
\end{theorem}
\begin{proof}

Define $q(t,a):= i(t,a) + r(t,a)$. From \eqref{pde}, $q$ satisfies the equation 
$$
\left( \frac{\partial}{\partial t} + \frac{\partial}{\partial a} \right) q(t,a) = \beta(a) s(t,a) B(t).
$$
We use then the method of characteristics to obtain that for $t>a$
\begin{align} \label{eq_ineq}
s(t,a) &= e^{\displaystyle -\int_0^a \beta(h) B(h+t-a)\ dh},\nonumber \\
q(t,a) &= \int_0^a \beta(h) B(h+t-a) s(h+t-a,h)\ dh.
\end{align}

Define $Q(a) := \displaystyle\limsup_{t\rightarrow \infty}\ q(t,a) $, $I(a) := \displaystyle\limsup_{t\rightarrow \infty}\ I(t,a) $ and $C := \int_0^\infty p_\infty(a) I(a)\ da$. Thus, from \eqref{eq_ineq} we deduce that
$$I(a)\leq Q(a)\leq C\int_0^a \beta(h) e^{-C\int_0^h \beta(\tau)\ d\tau}\ dh = 1-e^{-C\int_0^a \beta(\tau)\ d\tau}.$$
Multiplying by $p_\infty(a)$ and integrating with respect to $a$, we then have that
$$C \leq \int_0^a \left(1-e^{-C\int_0^a \beta(\tau)\ d\tau }\right) p_{\infty}(a)\ da.$$
Suppose that $C> 0$. We can write then
$$1 \leq \dfrac{1}{C} \int_0^a \left(1-e^{-C\int_0^a \beta(\tau)\ d\tau }\right) p_{\infty}(a)\ da =: F(C).$$

It is straightforward to verify that $$ \displaystyle \lim_{C\rightarrow 0} F(C) =  \mathcal{R}_C,$$ 
and that 
$$F'(C) = \dfrac{1}{C^2} \displaystyle \int_0^\infty p_\infty(a) e^{-C\int_0^a \beta(\tau)\ d\tau } \left(1+ C \int_0^a \beta(\tau)\ d\tau -e^{ C\int_0^a \beta(\tau)\ d\tau } \right)\ da < 0$$
since $e^x > 1+x$ for $x\neq 0$. Therefore, $$F(C)\leq \lim_{C\rightarrow 0} F(C) = \mathcal{R}_C < 1,$$
which is a contradiction. We then conclude that $C=0$. Since $i(t,a)\geq 0$ and $r(t,a) \geq 0$, we conclude that $\displaystyle\lim_{t\rightarrow \infty} i(t,a) = \displaystyle\lim_{t\rightarrow \infty} r(t,a) = 0$ and $\displaystyle\lim_{t\rightarrow \infty} s(t,a) = 1$.
\end{proof}
 
\section{Endemic non-uniform steady state distribution} \label{sec:steadyState}

In this section we analyze necessary conditions for the existence of non-uniform steady state distributions. We first prove that this is the case if $\mathcal{R}_0 > 1$. We then compute explicit conditions in order to observe a backward bifurcation, for the particular case of constant coefficients $\mu, \beta,\phi, \gamma, \rho, c$.

\subsection{Existence of at least one endemic non-uniform steady state}
In this section we establish conditions for which at least one positive solution exists. For our model, the existence of solutions mostly rely under the threshold condition $\mathcal{R}_0$.

\begin{theorem} \label{th*}
There exists an endemic non-uniform steady state of system \eqref{pde} when $\mathcal{R}_0 > 1$.
\end{theorem}
\begin{proof}
A non-uniform steady state age-distribution is a solution of the nonlinear system
\begin{subequations} \label{pde*}
\begin{align}
\begin{split} \label{PDE_s*}
\dfrac{ds^*(a)}{da}  &= - B^* \beta(a)  s^*(a),
\end{split}\\
\begin{split} \label{PDE_i*}
\dfrac{di^*(a)}{da} &= B^* \beta(a) s^*(a) -  (\phi(a)+\gamma(a)) i^*(a)+ B^*\rho(a) r^*(a) ,
\end{split}\\
\begin{split} \label{PDE_r*}
\dfrac{dr^*(a)}{da} &=  (\phi(a)+\gamma(a)) i^*(a) - B^*\rho(a) r^*(a) ,
\end{split}
\end{align}
with $a > 0$, along with boundary conditions
\begin{equation} \label{PDE_ic*}
s^*(0)=1,\quad  i^*(0)=0,\quad r^*(0)=0,
\end{equation}
where 
\begin{equation} \label{eq_B*}
B^* := \int_0^\infty i^*(a) p_\infty (a) da.
\end{equation}
\end{subequations}

Consider the linear system of equations with parameter $B$ given by 
\begin{subequations} \label{pde**}
\begin{align}
\begin{split} \label{PDE_s**}
\dfrac{ds_B^*(a)}{da}  &=  - B \beta(a)  s_B^*(a) ,
\end{split}\\
\begin{split} \label{PDE_i**}
\dfrac{di_B^*(a)}{da} &= B \beta(a) s_B^*(a) -( \phi(a)+\gamma(a)) i_B^*(a)+ B\rho(a) r_B^*(a) ,
\end{split}\\
\begin{split} \label{PDE_r**}
\dfrac{dr_B^*(a)}{da} &=  (\phi(a)+\gamma(a)) i_B^*(a) - B\rho(a) r_B^*(a) ,
\end{split}
\end{align}
along with boundary conditions
\begin{equation} \label{PDE_ic**}
s_B^*(0)=1,\quad  i_B^*(0)=0,\quad r_B^*(0)=0.
\end{equation}
\end{subequations}

Given the solution $(s_B^*(a), i_B^*(a), r_B^*(a))$ of system \eqref{pde**}, let
\begin{equation*}
H(B) := \int_0	^\infty i_B^*(a) p_\infty (a) da.
\end{equation*}
It is clear that $(s_B^*(a), i_B^*(a), r_B^*(a))$ satisfies system \eqref{pde*} if and only if $B$ is a fixed point of $H$; i.e., $H(B) = B$. Furthermore, if $B=0$ then 
$i_B^*(a) = 0$ and $H(0)=0$.
Thus, in order to guarantee existence of at least one non-trivial solution to \eqref{pde*}, it is just necessary to prove that $H(B)$ has a fixed point $B>0$. Consider the auxiliary function $\widehat{G}(B) = H(B)/B$ for $B\neq 0$. Clearly $\widehat{G}(B)$ is continuous in $[0,1]$ by defining $\widehat{G}(0) = H'(0)$. Since $H(1)<1$, we have that $\widehat{G}(1)<1$. Thus, we just need to prove that $\widehat{G}(0)>1$. If this is the case, then there exists $B\in (0,1)$ such that $\widehat{G}(B)=1$.

We then compute explicitly $\widehat{G}(B)$. First, from \eqref{PDE_s**}, we compute that
\begin{equation} \label{eq:sB*_sol}
s_B^*(a) = e^{-B\int_0^a\beta(h)dh} .
\end{equation}
Let $\bm{X}_B(a) = [i_B^*(a), r_B^*(a)]^T$. We can write \eqref{PDE_i**} and \eqref{PDE_r**} as the linear system
\begin{equation} \label{systemODES0}
\dfrac{d}{da	}\bm{X}_B(a) = M(a) \bm{X}_B(a)  + \bm{f}_B(a),
\end{equation}
where 
\begin{equation*}
M(a)=
  \begin{bmatrix}
    -(\phi(a)+\gamma(a)) & B\rho(a)\\
    \phi(a)+\gamma(a) & -B\rho(a)\\
  \end{bmatrix},\quad 
\bm{f}_B(a)=
  \begin{bmatrix}
    B\beta(a) s_B^*(a)\\
    0\\
  \end{bmatrix}.
\end{equation*}
We then observe that 
\begin{equation*}
\bm{X}^{(1)}_B(a) :=
\left[
\begin{array}{r}
1 \\ 
-1
\end{array} 
\right] e^{-\int_0^a (\phi(h)+\gamma(h)+ B\rho(h))dh}
\end{equation*}
is a solution of the homogeneous system $\dfrac{d}{da	}\bm{X}_B(a) = M(a) \bm{X}_B(a)$ with initial condition $\bm{X}^{(1)}_B(0) = [1, -1]^T$. By Abel's formula, we compute a second linear independent solution $\bm{X}^{(2)}_B(a)$ for the homogeneous system. After straightforward computations, a linear independent solution with initial condition $\bm{X}^{(2)}_B(0) = [0,1]^T$ is given by
\begin{equation*}
\bm{X}^{(2)}_B(a) :=
\left[
\begin{array}{r}
1 - e^{-\int_0^a (\phi(h)+\gamma(h)+ B\rho(h))dh} - \int_0^a (\phi(\tau)+\gamma(\tau)) e^{-\int_\tau^a  (\phi(h)+\gamma(h)+ B\rho(h))dh} d\tau  \\ 
e^{-\int_0^a (\phi(h)+\gamma(h)+ B\rho(h))dh}+\int_0^a (\phi(\tau)+\gamma(\tau)) e^{-\int_\tau^a  (\phi(h)+\gamma(h)+ B\rho(h))dh} d\tau
\end{array} 
\right].
\end{equation*}
Therefore, the fundamental matrix $\Phi(a)$ for system \eqref{systemODES0} is given by 
$$\Phi(a) = \left[\bm{X}^{(1)}_B(a)\quad \bm{X}^{(2)}_B(a) \right].$$
By variation of parameters, the solution of system \eqref{systemODES0} is given by $$\bm{X}_B(a)  = \Phi(a) \bm{C}(a),$$ where $\bm{C}(a)$ satisfies the equation 
\begin{equation} \label{eq:BCC}
\bm{C}'(a) = \Phi^{-1}(a) \bm{f}_B(a).
\end{equation}
Since 
$$\left[
\begin{array}{c}
0 \\ 
0
\end{array} 
\right] = \bm{X}_B(0)  = \Phi(0) \bm{C}(0) = 
\left[
\begin{array}{cc}
1 & 0 \\ 
-1 & 1
\end{array} 
\right] \bm{C}(0),$$
we deduce that $\bm{C}(0) = [0,0]^T$. Then, by integrating \eqref{eq:BCC} we obtain that 
\begin{align*}
\bm{C}(a) = \left[ \begin{array}{c}
C_1(a) \\ 
C_2(a)
\end{array} \right] = \left[
\begin{array}{c}
B\int_0^a   \beta(\tau) s_B^*(\tau) \left( 1+\int_0^\tau (\phi(t)+\gamma(t)) e^{\int_0^t (\phi(h)+\gamma(h)+B\rho(h))dh}\ dt \right) \ d\tau \\ 
B\int_0^a  \beta(\tau) s_B^*(\tau)\ d\tau
\end{array} 
\right].
\end{align*}
After some simplifications, we conclude that $r_B^*(a)$ and $i_B^*(a)$ are given explicitly by
\begin{align*}
r_B^*(a) = & \int_0^a (\phi(\tau) +\gamma(\tau))\left( 1- s_B^*(\tau )\right) e^{-\int_\tau^a ( \phi(h)+\gamma(h) + B \rho(h))\ dh}\ d\tau, \\
i_B^*(a) = &1 - s_B^*(a) - r_B^*(a).
\end{align*}
Since $\displaystyle \lim_{B\rightarrow 0}\dfrac{1-s_B(a)}{B} =  \int_0^a \beta (h)\ dh$,
we obtain that $$\widehat{G}(0) = \lim_{B\rightarrow 0} \widehat{G}(B) =  \int_0^\infty \int_0^a \beta(\tau) e^{-\int_\tau^a ( \phi(h)+\gamma(h) )dh} p_\infty (a)\ d\tau\ da = \mathcal{R}_0.$$
By hypothesis $ \mathcal{R}_0 >1 $. Therefore, there exists $B\in (0,1)$ such that $\widehat{G}(B) = 1$, completing our proof.
\end{proof}

\begin{remark} {\rm
In the case of constant coefficients, we can compute that $$\mathcal{R}_0 = \dfrac{\beta}{\mu+\phi+\gamma}, \quad \mathcal{R}_C = \dfrac{\beta}{\mu};$$ see \eqref{eq:RphiCC} below. 
Thus, a large value for $\beta$ implies the existence of at least one endemic non-uniform steady state, while a small value for $\beta$ guarantees that the disease tends to zero. In the next section we explore necessary and sufficient conditions that lead to the existence of a backward bifurcation.}
\end{remark}


\subsection{Conditions for a backward bifurcation}
In this section we explore conditions for the existence of two positive steady states when $\mathcal{R}_0<1$. For simplifity, it is assumed that $\mu$, $\beta$, $\rho$, $\phi$, $\gamma$ and $c$ are constant in the characterization of the backward bifurcation. 

\begin{remark} {\rm
We assume constant parameters as an illustrative example, however, initial conditions remain age-dependant.
}
\end{remark}

By solving \eqref{PDE_s**} we obtain that (see \eqref{eq:sB*_sol})
\begin{equation*}
s_B^*(a) := e^{-aB \beta } .
\end{equation*}
Moreover, \eqref{systemODES0} becomes a linear system with constant coefficients given by
\begin{equation} \label{systemODES}
\dfrac{d}{da	}\bm{X}_B(a) = M \bm{X}_B(a)  + \bm{f}_B(a),
\end{equation}
where 
\begin{equation*}
M=
  \begin{bmatrix}
    -(\phi+\gamma) & B\rho\\
    \phi+\gamma & -B\rho\\
  \end{bmatrix},\quad 
\bm{f}_B(a)=
  \begin{bmatrix}
    B\beta s_B^*(a)\\
    0\\
  \end{bmatrix}.
\end{equation*}
Observe that the eigenvalues of $M$ are given by $\lambda_1 = 0$ and $\lambda
_2 = - (\phi+\gamma) - B \rho$, with corresponding eigenvectors $\bm{v}_1 = [B\rho, \phi+\gamma]^T$, $\bm{v}_2 = [-1, 1]^T$. Then, the solution of \eqref{systemODES} is given by
\begin{align*}
i_B^*(a) =\ &\dfrac{ B\rho}{\phi+\gamma+B\rho}-\dfrac{B(\rho-\beta)}{(\phi+\gamma+B(\rho-\beta))}\cdot e^{-B\beta a}\nonumber \\
&-\dfrac{B \beta (\phi+\gamma) }{(\phi+\gamma+B\rho)(\phi+\gamma+B(\rho-\beta))}\cdot e^{-(\phi+\gamma+B\rho)a},
\end{align*}
and 
\begin{align*}
r_B^*(a) =\ &\dfrac{ \phi+\gamma}{\phi+\gamma+B\rho}-\dfrac{\phi+\gamma}{\phi+\gamma+B(\rho-\beta)}\cdot e^{-B\beta a}\nonumber \\
&+\dfrac{B \beta (\phi+\gamma) }{(\phi+\gamma+B\rho)(\phi+\gamma+B(\rho-\beta))}\cdot e^{-(\phi+\gamma+B\rho)a}.
\end{align*}
Note that $$p_\infty(a) = \lim_{t\rightarrow\infty} p(t,a)=\frac{\mathcal{F}(a)}{\int_0^\infty  \mathcal{F}(a)\ da} = \mu e^{-\mu a}.$$ 
After some simplifications, we obtain that 
\begin{align*} 
\widehat{G}(B) = \dfrac{B+\dfrac{\mu}{\rho}}{\left(B+\dfrac{\mu}{\beta}\right)\left( B+\dfrac{\mu+\phi+\gamma}{\rho}\right)}.
\end{align*}
In particular, the basic reproductive number is given by
\begin{equation} 	\label{eq:RphiCC}
\mathcal{R}_0 = \widehat{G}(0) = \dfrac{\beta }{\mu +\phi+\gamma}\ \quad \text{and }\ \quad \mathcal{R}_C = \dfrac{\beta}{\mu}.
\end{equation}

It can be proven that the equation $\widehat{G}(B)=1$ has two real roots $B\in (0,1)$ if and only if the following conditions are satisfied:
\begin{align} \label{eq_cond_bif}
&\mathcal{R}_0 <1 < \mathcal{R}_C,\nonumber \\ 
&\beta > \frac{\rho\mu}{\mu+(\sqrt{\rho}-\sqrt{\phi+\gamma})^2}. 
\end{align}

\begin{figure}[htb!]
\centering
\subfloat[Parameter space $\left(\frac{\phi+\gamma}{\mu},\frac{\beta}{\mu},\frac{\rho}{\mu}\right)$. \label{fig:condCoef}]{\includegraphics[width=0.45\linewidth]{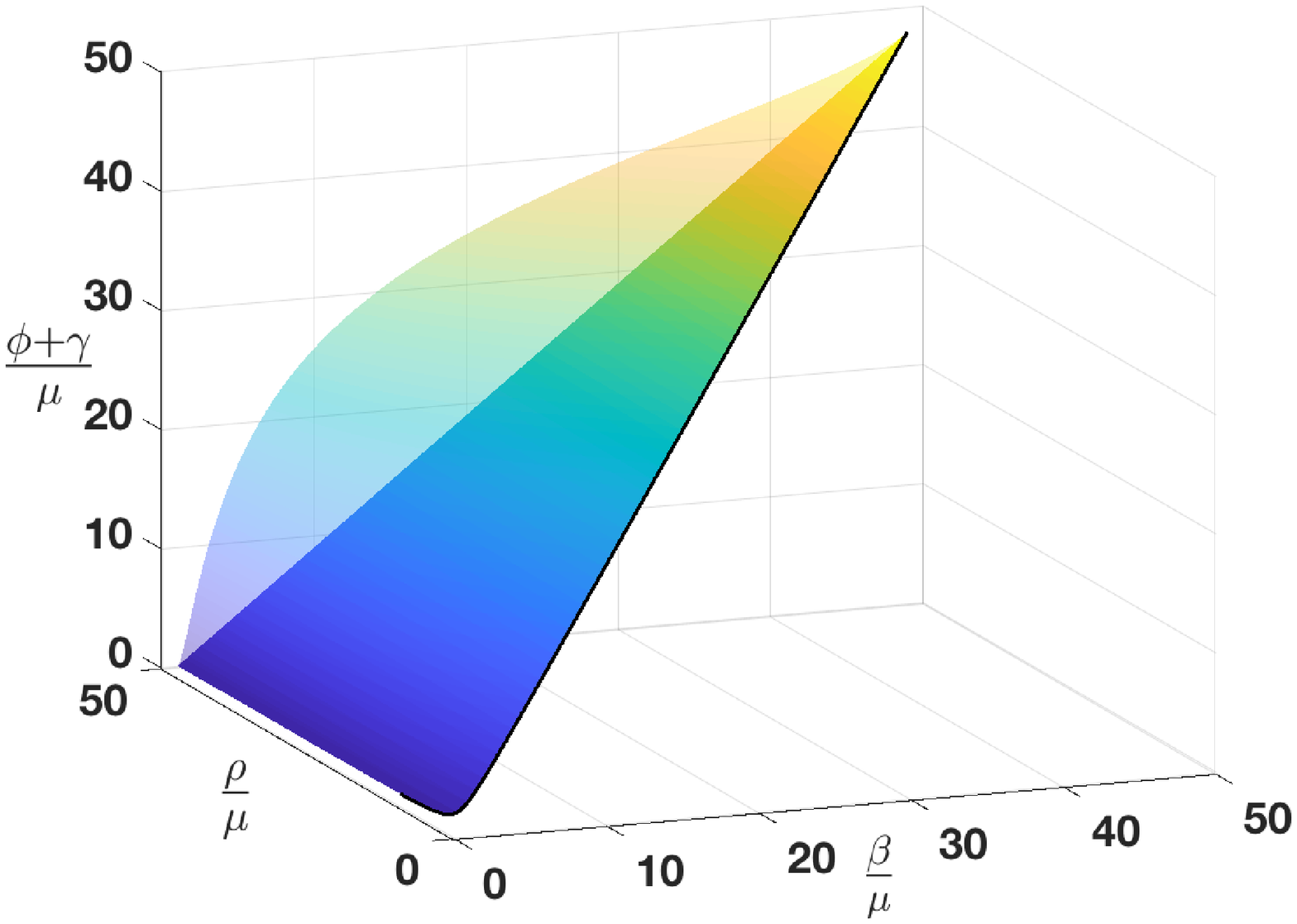}}
\hfill
\subfloat[$\widehat{G}(B)$ vs B. \label{fig:bif}]{\includegraphics[width=0.45\linewidth]{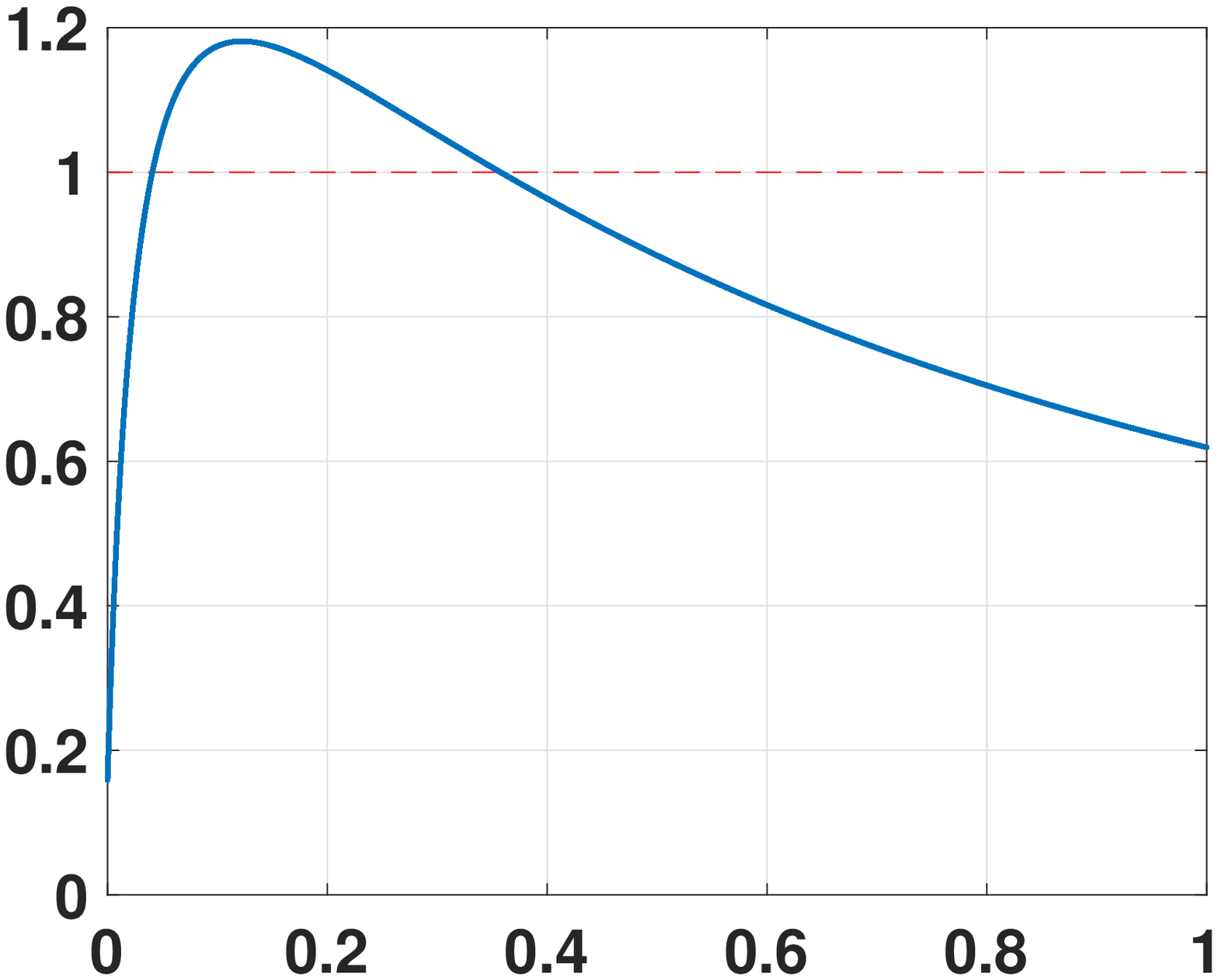}}
\caption{(a) Domain for which two non-zero solutions exist according to \eqref{eq_cond_bif}. (b) $\widehat{G}(B)$ for parameters satisfying \eqref{eq_cond_bif}; solutions to $\widehat{G}(B)=1$ correspond to non-trivial steady-state distributions for system \eqref{pde*}; see also \Cref{fig:i*}. \label{fig:GB}}
\end{figure} 

\begin{figure}[htb!]
\centering
\subfloat[$B^*\approx 0.04072$ \label{fig:i*a}]{\includegraphics[width=.45\linewidth]{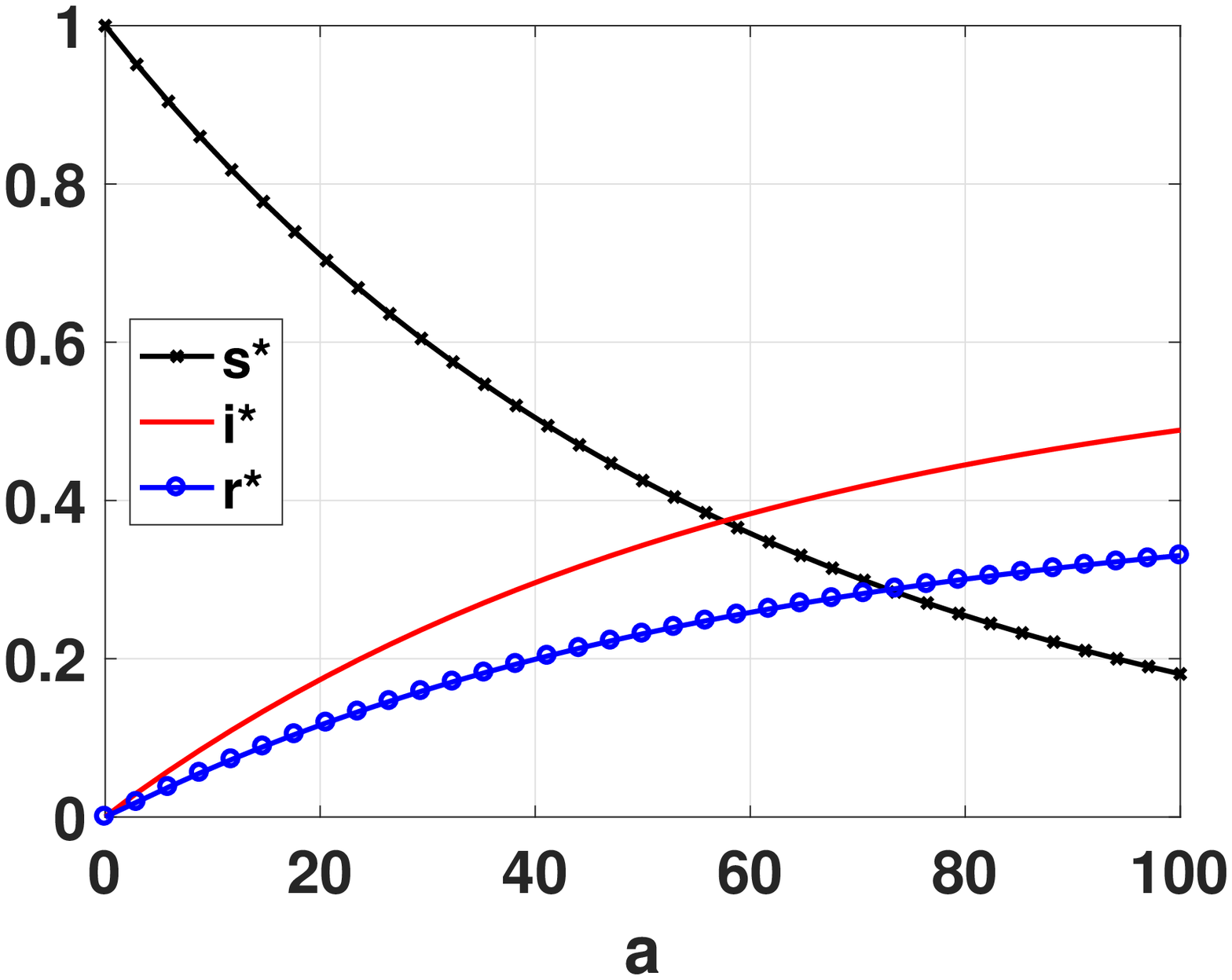}}
\hfill
\subfloat[$B^*\approx 0.35754$ \label{fig:i*b}]{\includegraphics[width=.45\linewidth]{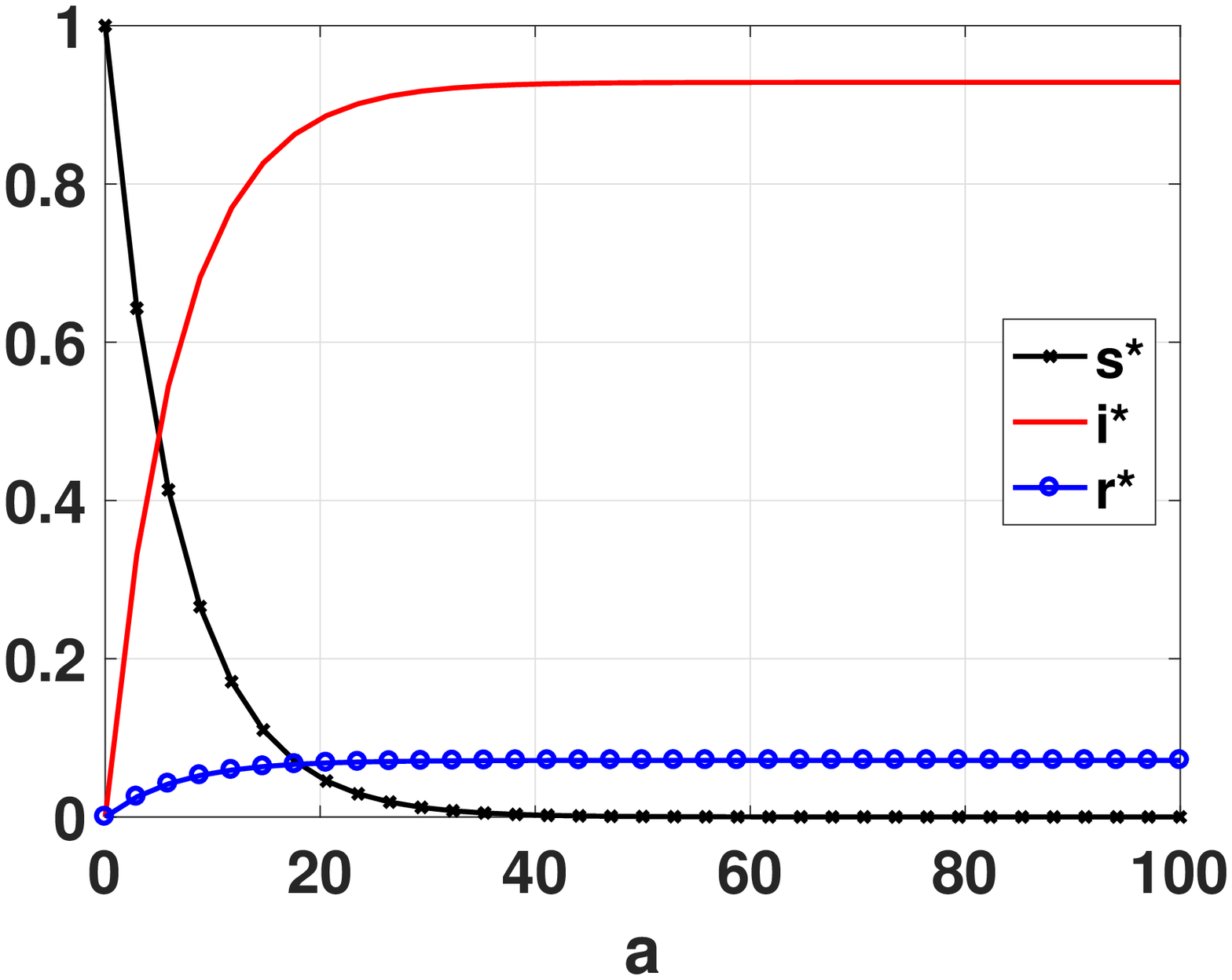}}
\caption{Two non-trivial solutions $(s^*(a),i^*(a),r^*(a))$ for system \eqref{pde*} with constant parameters.  \label{fig:i*}}
\end{figure} 

These conditions can be represented graphically as the interior of the volume depicted in \Cref{fig:condCoef} in the space $\left(\frac{\phi+\gamma}{\mu},\frac{\beta}{\mu},\frac{\rho}{\mu}\right)$. In \Cref{fig:bif} we present a typical graph for $\widehat{G}(B)$ with constant parameters $(\phi,\gamma,\beta,\rho,\mu)$ satisfying \eqref{eq_cond_bif}. It is clear that $\widehat{G}(B) = B$ has two different solutions, and for each value there is a non-trivial solution $(s^*(a),i^*(a),r^*(a))$ as shown in \Cref{fig:i*}. We remark that conditions \eqref{eq_cond_bif} are necessary and sufficient in order to guarantee existence of two non-zero steady states.

\begin{figure}[bt!]
   \centering
   \includegraphics[width=.6\textwidth]{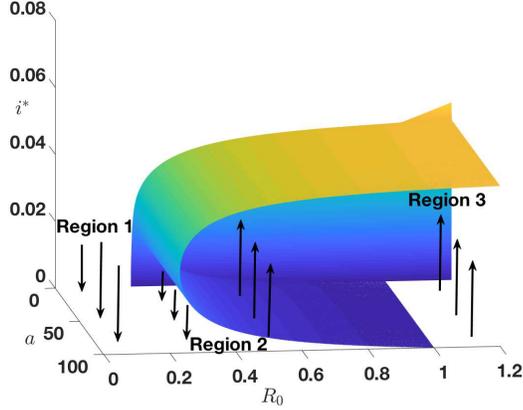}
   \caption{When conditions given in~\eqref{eq_cond_bif} hold we observe two non-zero solutions as shown in region 2 of the backward bifurcation. Parameter values are: $\mu=0.0125$, $\phi=60$, $\gamma=13$, $\rho=76.65$ as in Example \ref{sec:exparams}. \label{fig:backbif}}
\end{figure} 

In \Cref{fig:backbif} we show a backward bifurcation obtained by fixing $\mu, \rho, \phi, \gamma$ as given in \eqref{eq:coef:bif} and varying $\beta$. In this case, we have infected (infectious) individuals $i^*$ as a function of age $a$ and $\mathcal{R}_0 = \frac{\beta}{\mu+\phi+\gamma}$. This is typical behavior when we have a backward bifurcation and $\mathcal{R}_0<1$, which implies $\mathcal{R}_0<1$ is not a sufficient condition to control the disease. Hence, the condition of $\mathcal{R}_0 \ll 1$ is necessary to reach a infection-free state (see Region 1 in \Cref{fig:backbif}). When conditions \eqref{eq_cond_bif} are met, we observe two non-trivial solutions for $i^*(a)$ (see region 2). This implies that under the scenario where there is a critical mass of infectious individuals in a population, it would make the control of the disease more difficult. This is relevant for entities responsible to develop prevention/control strategies. For $\mathcal{R}_0>1$, \Cref{th*} implies the existence of a non-uniform steady state distribution (see Region 3). 

\section{Numerical experiments} \label{sec:exp}
In this section, we first describe the numerical implementation for solving system \eqref{pde}. We then present different examples confirming our theoretical results obtained in \Cref{sec:freesteadyState} and \Cref{sec:steadyState}. All parameters are in units of $year^{-1}$. After analyzing the model in its general form, we implement a numerical example based on the analogous nonlinear ordinary differential equation model in~\cite{sanchez2007}. That is, we look at the nonlinear dynamics of drinkers as an \lq\lq epidemic\rq\rq process where individuals who are considered {\it problem drinkers} can cause other individuals to start drinking. Eventually, these individuals can {\it temporarily} recover but as state in \cite{finney1999relapse,jin1998relapse,moss2007relapse}, relapse rates are high and the probability of never drinking again is small.

\subsection{Numerical implementation} \label{sec:numer_imp}
We will discretize \eqref{pde} with a first-order upwind finite difference scheme and will approximate the solution on the physical domain of interest given by the rectangle $\lbrace (t,a)\in [0,T]\times [0,A]\rbrace$. We first construct a uniform grid with equidistant points. Divide $[0,T]$ and $[0,A]$ into $N_T$ and $N_A$ subintervals, respectively. Thus, the nodes $(t_j,a_k)$ on the rectangular mesh are given by 
$$(t_j,a_k) = \left( j\Delta t, k\Delta a\right),$$
for $j\in\lbrace 0,1,\ldots,N_T\rbrace$, $k\in\lbrace 0,1,\ldots,N_A\rbrace$, where 
$$\Delta t := \frac{T}{N_T},\quad \Delta a := \frac{A}{N_A}$$ are the corresponding step sizes. We require that $\Delta t <\Delta a$ in order to satisfy the CFL and stability conditions of the scheme.

For any function $x$ and a grid point $(t_j,a_k)$, we denote the approximation of $x(t_j,a_k)$ by $x_k^j$. 
Since
\begin{align*}
\left( \frac{\partial}{\partial t} + \frac{\partial}{\partial a} \right) x(t_j,a_k) =&\  \frac{x(t_j+\Delta t, a_k)-x(t_j,a_k)}{\Delta t} + \mathcal{O}(\Delta t) \\
& \quad +  \frac{x(t_j,a_k)- x(t_j,a_{k}-\Delta a)}{\Delta a} + \mathcal{O}(\Delta a),
\end{align*}
we approximate the derivatives by
\begin{equation*}
\left( \frac{\partial}{\partial t} + \frac{\partial}{\partial a} \right) x(t_j,a_k) \approx  \frac{x^{j+1}_{k} - x^{j}_{k}}{\Delta t} + \frac{x^j_{k} - x^{j}_{k-1}}{\Delta a}.
\end{equation*}
Similarly, define $\beta_k:=\beta(a_k)$, $\phi_k:=\phi(a_k)$,  $\gamma_k:=\gamma(a_k)$, $\rho_k:=\rho(a_k)$ and $B^j:=B(t_j)$. Then, by evaluating at all the grid points, the discretization for system \eqref{pde} is given by the explicit system
\begin{align} \label{PDE:discrete}
\frac{s^{j+1}_{k} - s^{j}_k}{\Delta t} + \frac{s^j_k - s^{j}_{k-1}}{\Delta a} &= - \beta_k s^j_k B^j,\nonumber\\
\frac{i^{j+1}_k - i^{j}_k}{\Delta t} + \frac{i^j_k - i^{j}_{k-1}}{\Delta a} &= \beta_k s^j_k B^j - (\phi_k+\gamma_k) i^j_k+\rho_k r_k^j B^j , \nonumber \\
\frac{r^{j+1}_k - r^{j}_k}{\Delta t} + \frac{r^j_k - r^{j}_{k-1}}{\Delta a} &= (\phi_k+\gamma_k) i^j_k-\rho_k r_k^j B^j,
\end{align}
for $1\leq k\leq N_A$ and $0\leq j\leq N_T-1$. We recall that the initial conditions \eqref{pde_ic} provide the values $s_0^j$, $s_k^0$, $i_0^j$, $i_k^0$, $r_0^j$, and $r_k^0$. The integral $$B^j = B(t_j) = \int_0^\infty i(t_j,a) p(t_j,a)\ da$$ is approximated via MATLAB's command \texttt{integral}. We note that for a fixed time $t_j$, $B^j$ depends only on values of $i$ at the same time $t_j$.

Solving for $s^{j+1}_k$, $i^{j+1}_k$ and $r^{j+1}_k$ from \eqref{PDE:discrete}, we obtain that
\begin{subequations}
\begin{align*}
\begin{split} 
s^{j+1}_k  &= s^{j}_k + \Delta t \left(- \beta_k s^j_k B^j  - \frac{s^j_k - s^{j}_{k-1}}{\Delta a}\right),
\end{split}\\
\begin{split}
i^{j+1}_k &= i^{j}_k +  \Delta t\left( \beta_k s^j_k B^j - (\phi_k+\gamma_k) i^j_k+\rho_k r_k^j B^j -  \frac{i^j_k - i^{j}_{k-1}}{\Delta a} \right),
\end{split}\\
\begin{split} 
r^{j+1}_k &= r^{j}_k+ \Delta t\left( (\phi_k+\gamma_k) i^j_k-\rho_k r_k^j B^j  -  \frac{r^j_k - r^{j}_{k-1}}{\Delta a} \right),
\end{split}
\end{align*}
\end{subequations}
for $1\leq k\leq N_A$ and $0\leq j\leq N_T-1$. Therefore, starting with the initial conditions $s_0^j$, $s_k^0$, $i_0^j$, $i_k^0$, $r_0^j$, $r_k^0$, we can compute, in successive time steps, the values of the unknowns on the grid points. The convergence of the scheme has been established in \cite{shim2006}.

\subsection{Global stability} \label{ex:R0lt1}
In this example we confirm the result shown in \Cref{th:globalStab}, where global stability 
to the infection-free state
is guaranteed when $\mathcal{R}_C<1$. We consider constant coefficients given by
\begin{equation*}
\begin{array}{ccccc}
\mu = 0.0125,&\beta= 0.011,& \phi= 60,& \gamma=13,& \rho=76.65.
\end{array}
\end{equation*}
Thus, $\mathcal{R}_0 \approx  1.5\cdot 10^{-4}$ and $\mathcal{R}_C = 0.88$. We then obtain $i(t,a)$ as shown in \Cref{fig:i_r0_lt1}. The parameter values used in our model are based on ~\cite{sanchez2007}.

\begin{figure}[t!]
\centering
\subfloat[Solutions of the system go to zero. \label{fig:bifurcation-1}]{\includegraphics[width=.45\linewidth]{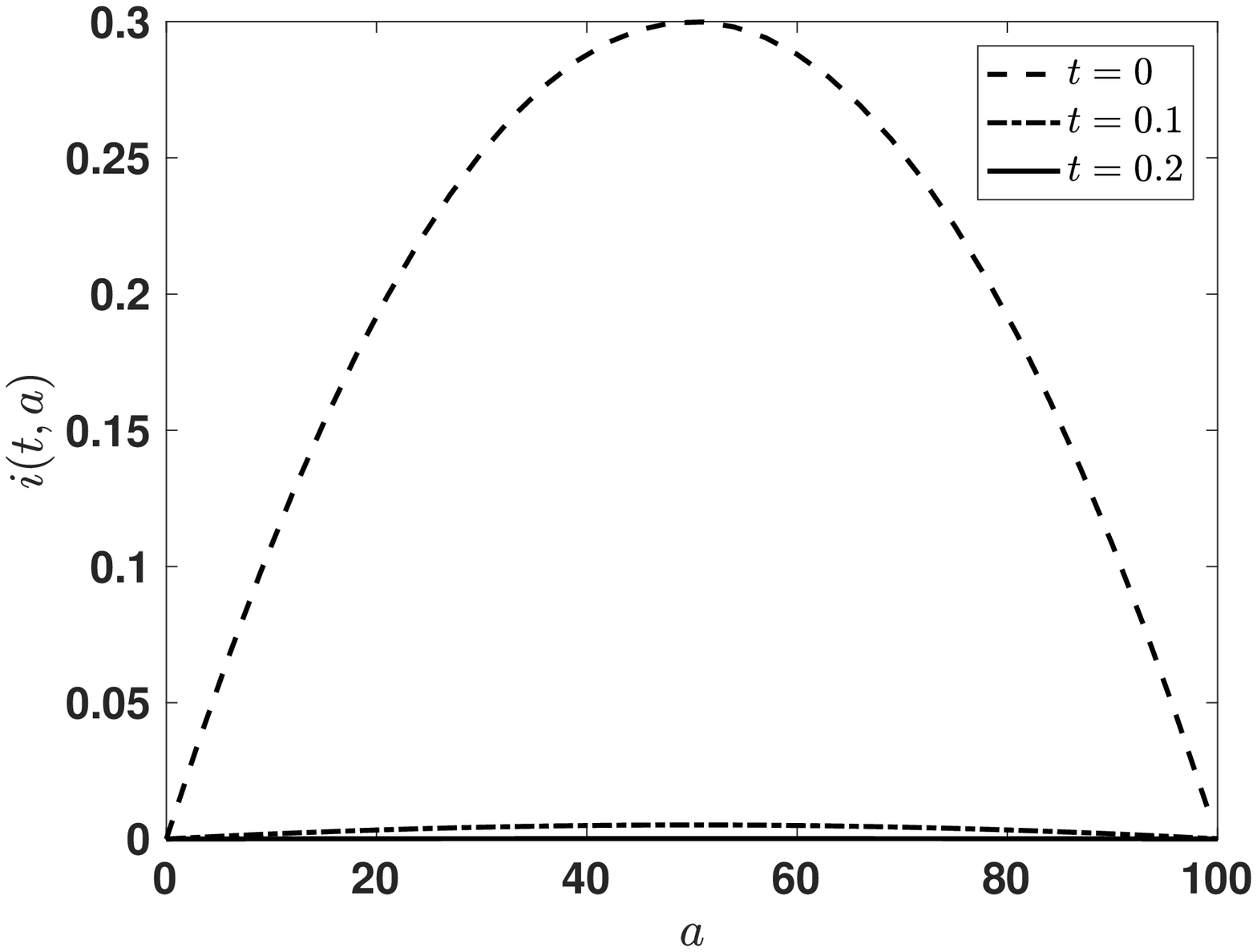}}     
\hfill
\subfloat[Solution for $i(t,a)$. \label{fig:i_r0_lt1}]{\includegraphics[width=.45\linewidth]{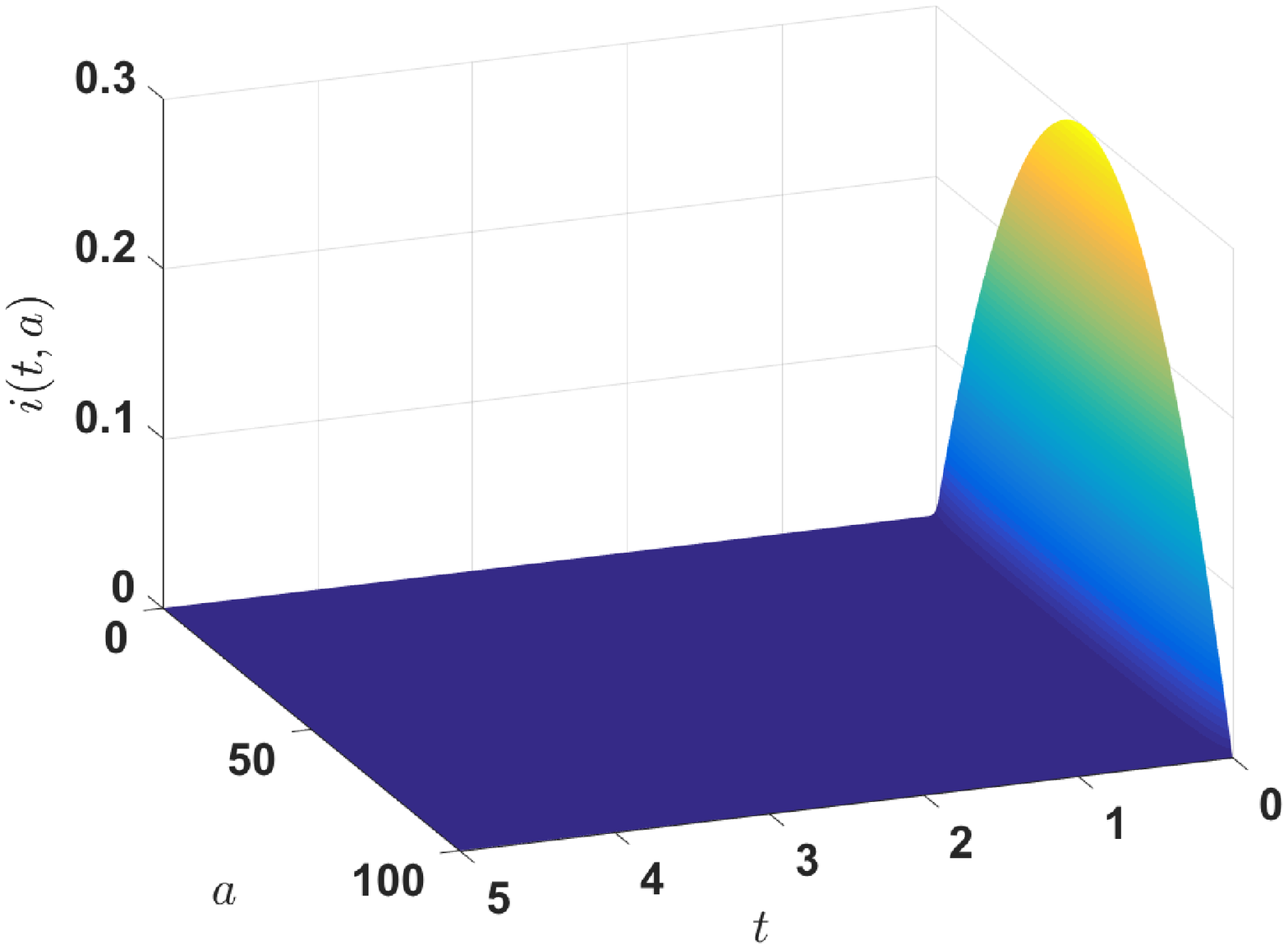}}
\caption{Numerical solution $i(t,a)$ with fixed parameters given in \eqref{eq:coef:bif} where $\mathcal{R}_0<1$ and $\mathcal{R}_C<1$.}
\end{figure} 

For the given set of parameters all solutions for $i(t,a)$ go to zero as $t\rightarrow\infty$. In this case, we are in the first region of the backward bifurcation (see \Cref{fig:backbif}) where $\mathcal{R}_0 \ll 1$. Here $\beta<\rho$ and $\beta$ is not sufficiently large to sustain enough drinkers in the system.

\subsection{Backward bifurcation} \label{sec:exparams}
We now consider a particular choice of parameters that satisfy conditions \eqref{eq_cond_bif}. In particular, we fix 
\begin{equation} \label{eq:coef:bif}
\begin{array}{ccccc}
\mu = 0.0125,&\beta= 60,& \phi= 60,& \gamma=13,& \rho=76.65.
\end{array}
\end{equation}
For this case, 
\begin{equation*}
\mathcal{R}_0 \approx 0.8218\quad {\rm and }\quad \mathcal{R}_C = 4800.
\end{equation*}

We are now in region 2 of \Cref{fig:backbif}, where we need $\mathcal{R}_0<1<\mathcal{R}_C$. We then have two non-trivial steady-state solutions for system \eqref{pde*} depicted in \Cref{fig:i*}. We then obtain $i(t,a)$ as described in \Cref{sec:numer_imp} for two different initial conditions $i(0,a)$ ; see \Cref{fig:sol2}. 


 
\begin{figure}[t!]
\subfloat[Solutions of the system go to endemic state when initial conditions are large enough.\label{fig:sol2a1}]{\includegraphics[width=.45\linewidth]{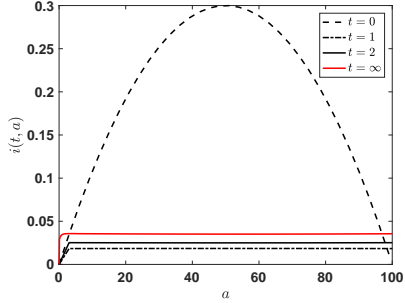}}
\hfill
\subfloat[Time and age series of $i(t,a)$.\label{fig:sol2a}]{ \includegraphics[width=.45\linewidth]{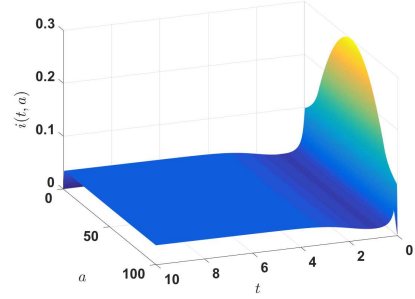}}                   
\hfill
\subfloat[Solutions of the system go to infection-free state when initial conditions are sufficiently small.\label{fig:sol2a2}]{\includegraphics[width=.45\linewidth]{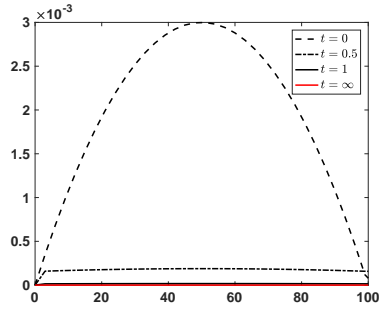}}
\hfill
\subfloat[Time and age series of $i(t,a)$.\label{fig:sol2b}]{\includegraphics[width=.45\linewidth]{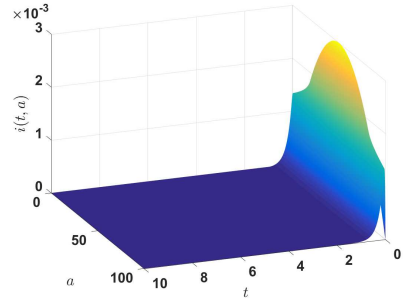}}
\caption{Numerical solution $i(t,a)$ with fixed parameters given in \eqref{eq:coef:bif} and two different initial conditions $i(0,a)$ when $\mathcal{R}_0<1$ and $\mathcal{R}_C>1$. \label{fig:sol2}}
\end{figure} 

In \Cref{fig:sol2a} we obtain the stable solution that corresponds to the steady-state distribution shown in \Cref{fig:i*b}. For a smaller initial condition $i(0,a)$, we confirm that the steady-state distribution shown in \Cref{fig:i*a} is unstable, since $i(t,a)$ tends to the (stable) zero solution; see \Cref{fig:sol2b}. In this case, albeit having $\beta<\rho$, $\beta$ is sufficiently large to sustain multiple drinking steady states dependent on the initial state of the drinking population. That is, $\mathcal{R}_0<1$ is not a sufficient condition for the stability of the {\it drinking-free} state. 

%

\subsection{Endemic non-uniform steady state distribution}  \label{ex:Rgt1}
According to \Cref{th*}, there exists a non-uniform steady state distribution if $\mathcal{R}_0>1$. In this example we take
\begin{equation*}
\begin{array}{ccccc}
\mu = 0.0125,&\beta= 120,& \phi= 60,& \gamma=13,& \rho=76.65,
\end{array}
\end{equation*}
for which $\mathcal{R}_0 \approx  1.6436$.
We obtain $i(t,a)$ as shown in \Cref{fig:ex23}. It is clear that $i(t,a)$ converges to a non-zero solution.
\begin{figure}[htb]
\centering
\subfloat[Solutions of the system go to endemic state.\label{fig:bifurcation-3}]{\includegraphics[width=.45\linewidth]{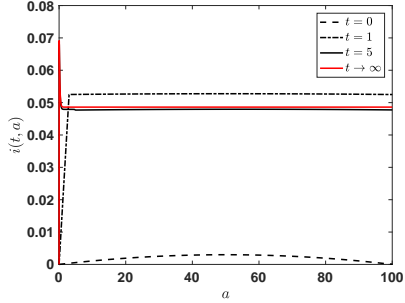}}
\hfill
\subfloat[Time and age series of $i(t,a)$.\label{fig:i_r0_lt1b}]{\includegraphics[width=.45\linewidth]{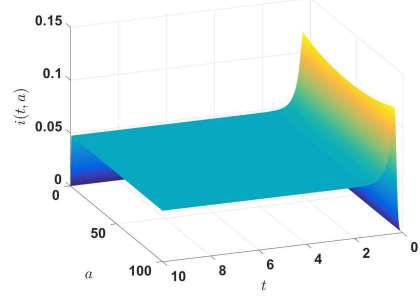}}
\caption{Numerical solution $i(t,a)$ when $\mathcal{R}_0>1$ and $\mathcal{R}_C>1$. \label{fig:ex23}}
\end{figure} 

Here $\beta>\rho$ and we have $\mathcal{R}_0>1$, which is sufficient for the drinking steady state to prevail and the initial conditions do not play a role (see \Cref{fig:backbif}).

\subsection{Numerical experiments with age-dependent parameters} \label{ex:var_coef}
Lastly, we present an example where the parameters $\phi(a)$, $\gamma(a)$ $\beta(a)$, $\rho(a)$, $\mu(a)$ and $c(a)$ are age-dependent with particular distributions; see \Cref{fig:var_coef:}. These were chosen arbitrarily, however, we pursued biological significance when choosing each parameter distribution. 

\begin{figure}[th!]
	\centering
	\includegraphics[width=0.55\textwidth]{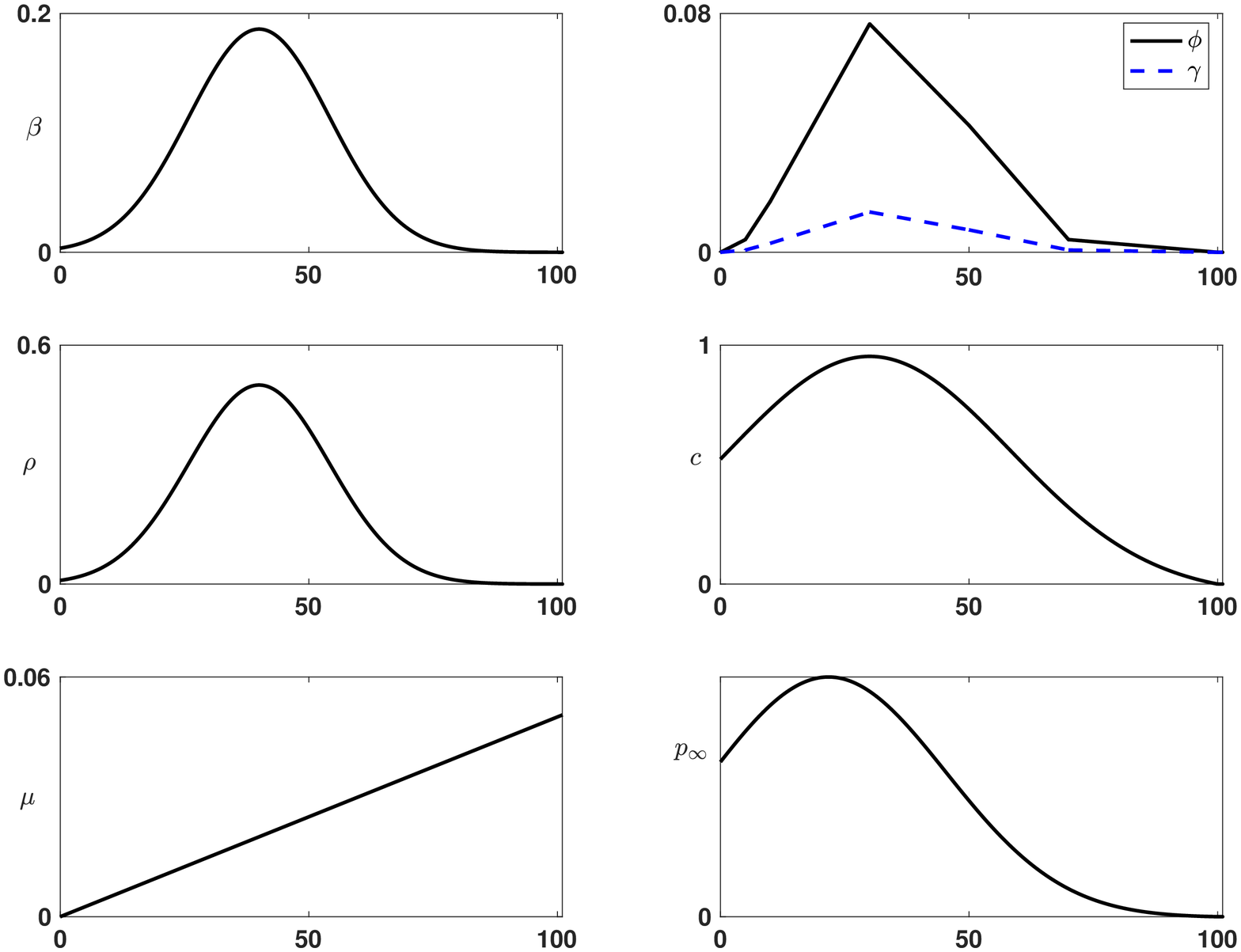}
      \caption{Age-dependent coefficients considered in Example \ref{ex:var_coef}. In particular we chose $\rho(a)$ to have a higher value than $\beta(a)$ for certain age groups given that research shows that individuals who have been been alcoholics are prone to relapse at a higher rate than individuals who have never been alcoholics~\cite{finney1999relapse,jin1998relapse,moss2007relapse}. \label{fig:var_coef:}}
\end{figure}

\begin{figure}[t!]
\centering
\subfloat[$\mathcal{R}_0<1$, $\mathcal{R}_C<1$. \label{fig:var_coef_reg1}]{\includegraphics[width=.45\linewidth]{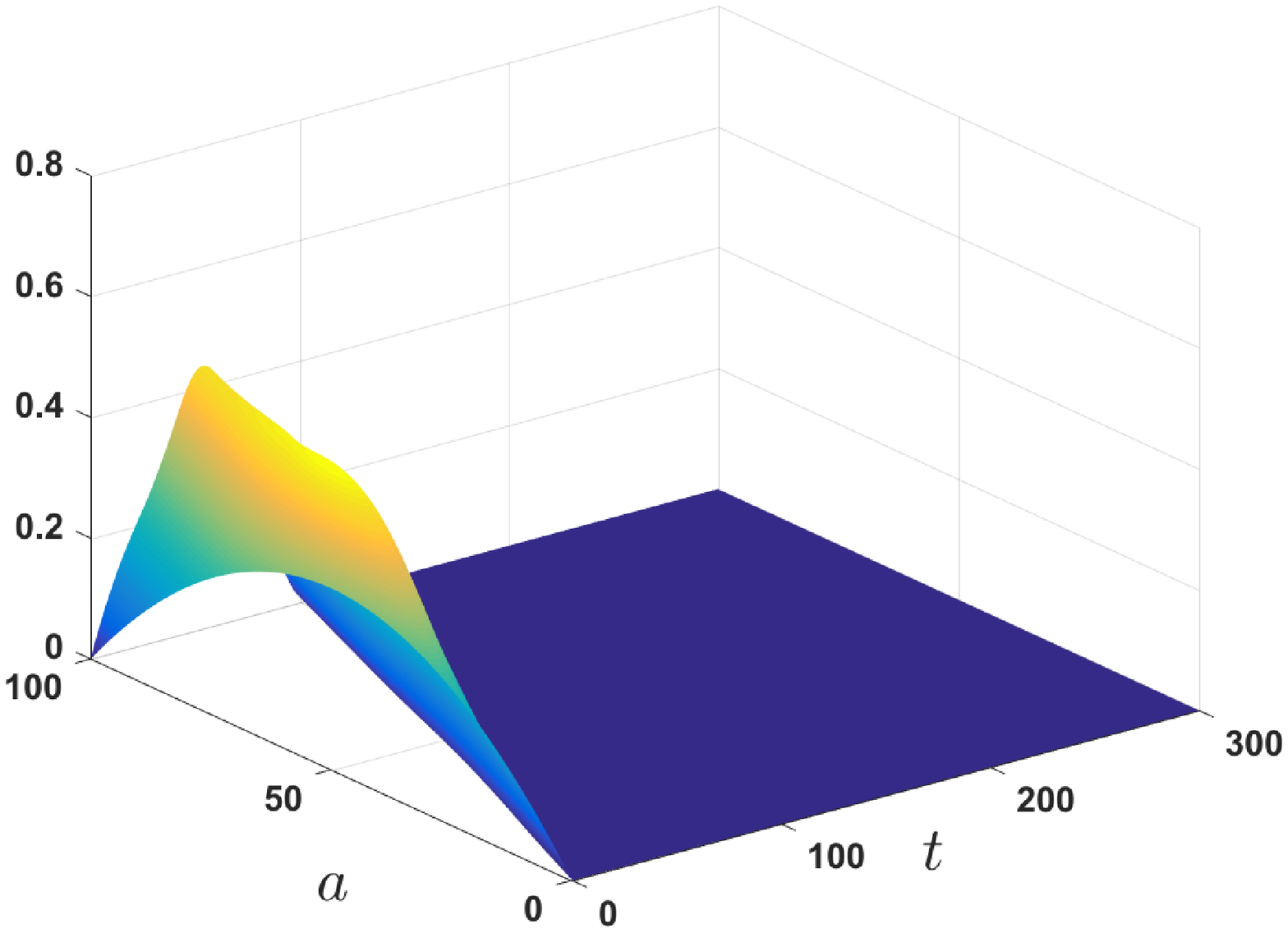}}
\hfill
\subfloat[$\mathcal{R}_0>1$, $\mathcal{R}_C>1$. \label{fig:var_coef_reg3}]{\includegraphics[width=.45\linewidth]{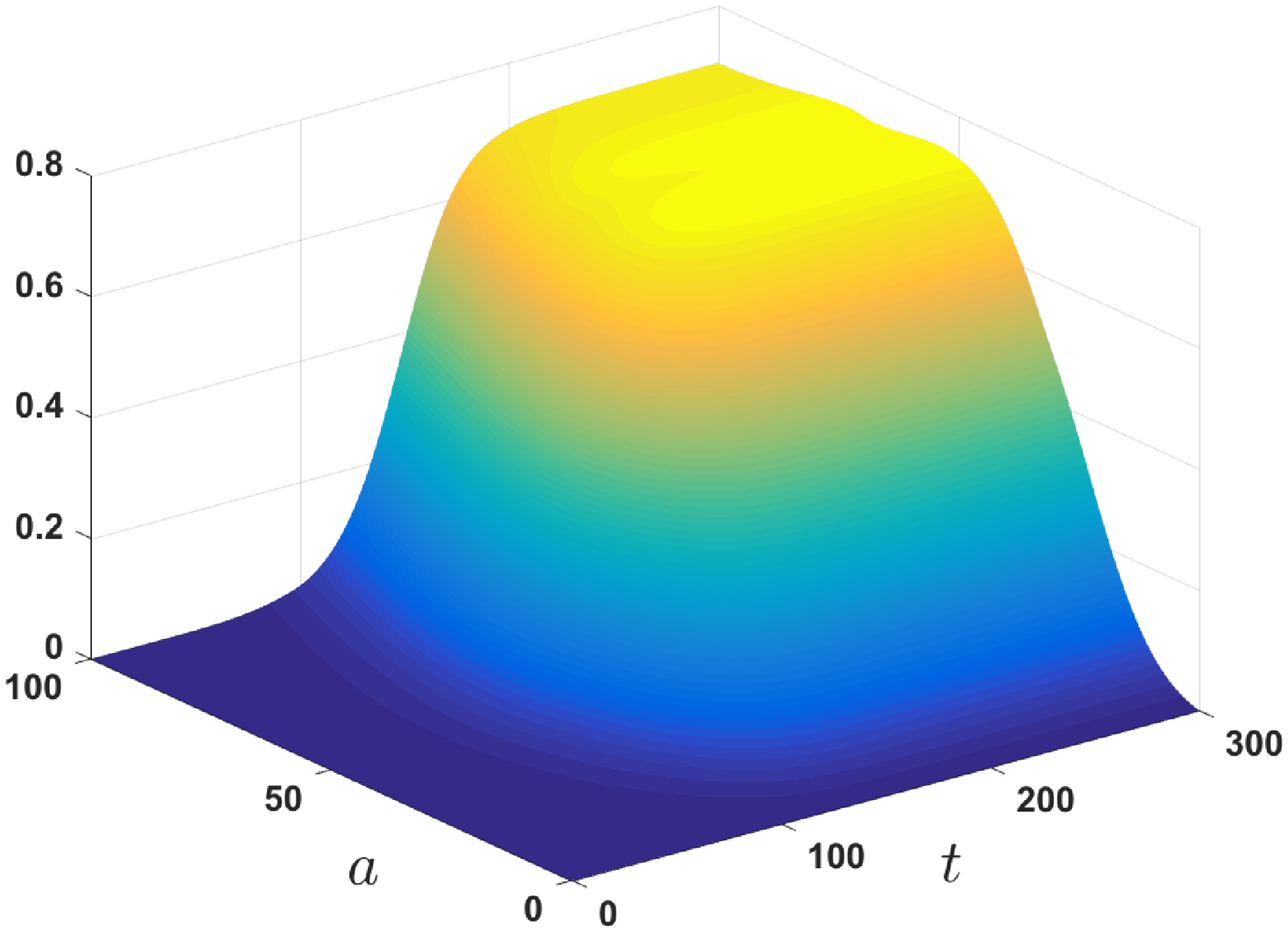}}
\hfill
\subfloat[$\mathcal{R}_0<1$, $\mathcal{R}_C>1$.\label{fig:var_coef_reg2a}]{\includegraphics[width=.45\linewidth]{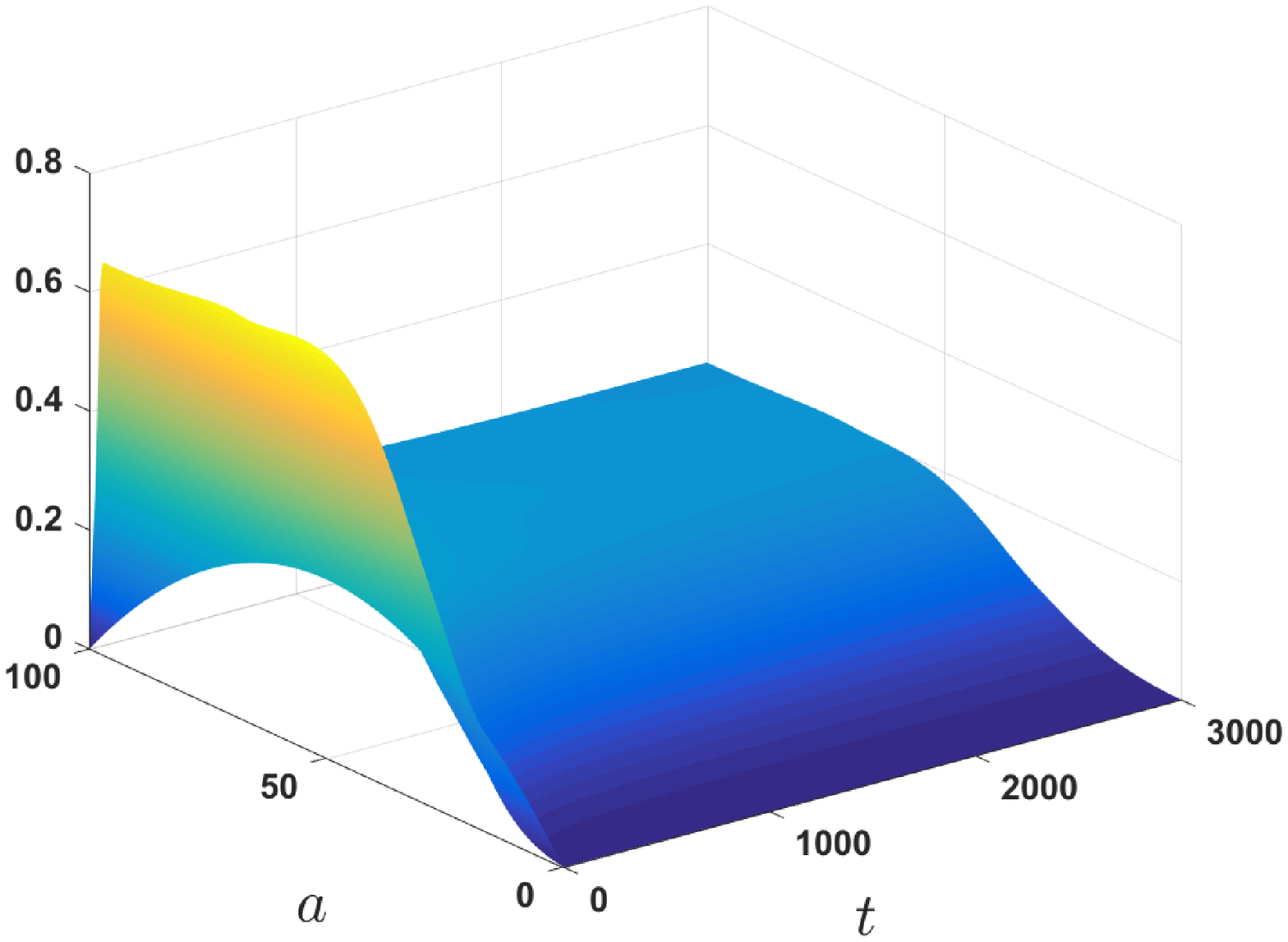}}
\hfill
\subfloat[$\mathcal{R}_0<1$, $\mathcal{R}_C>1$.\label{fig:var_coef_reg2b}]{\includegraphics[width=.45\linewidth]{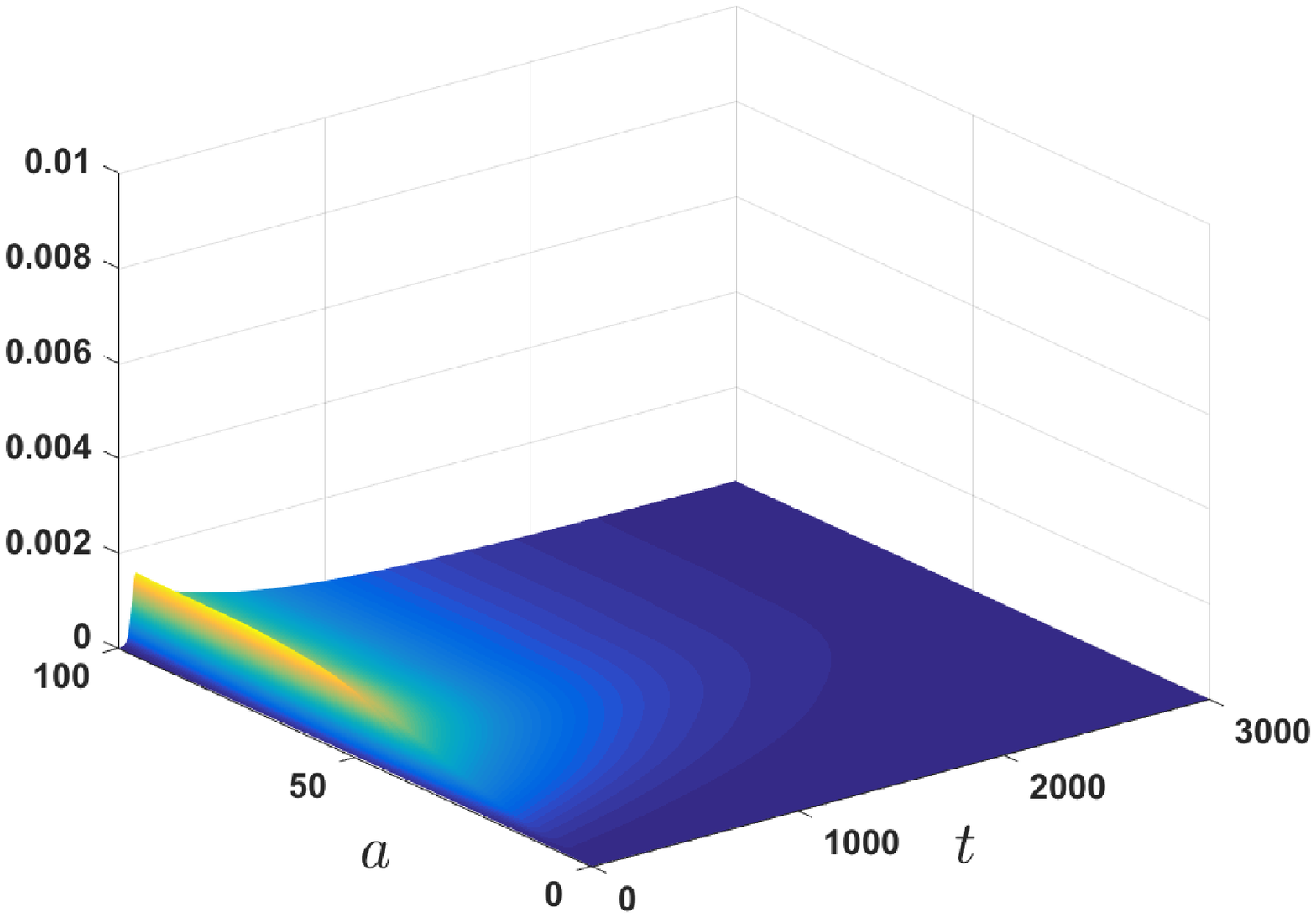}}
\caption{Numerical solution $i(t,a)$ with age-dependent parameters $\phi(a)$, $\gamma(a)$, $\beta(a)$, $\rho(a)$ and different values of $\mathcal{R}_0$ and $\mathcal{R}_C$. Two different initial conditions $i(0,a)$ when $\mathcal{R}_0<1$ and $\mathcal{R}_C>1$ are used in  \Cref{fig:var_coef_reg2a,fig:var_coef_reg2b}; see Example \ref{ex:var_coef}. \label{fig:var_coef}}
\end{figure} 

\begin{figure}[thb!]
	\centering
	\includegraphics[width=0.6\textwidth]{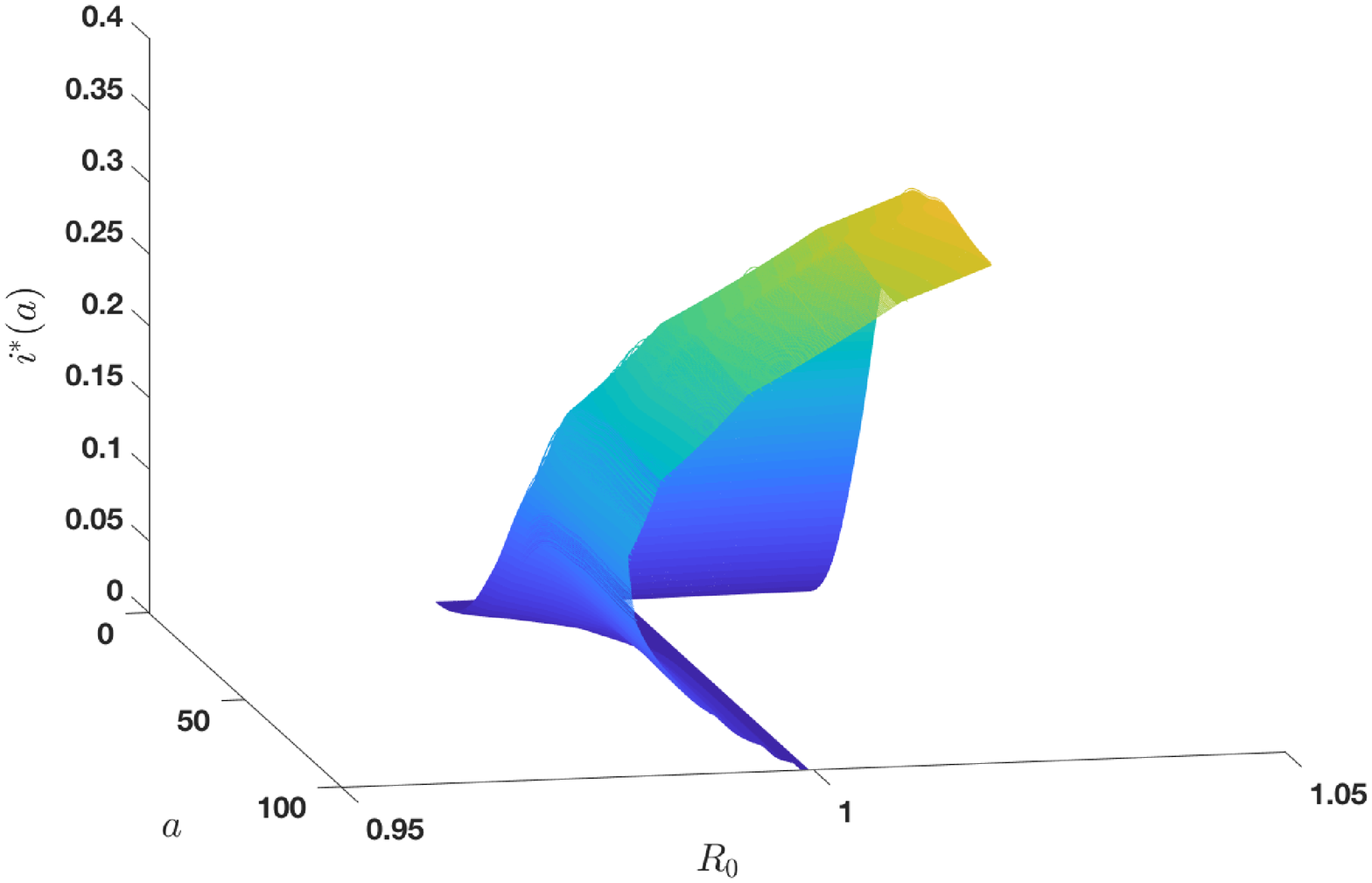}
      \caption{Backward bifurcation with age-dependent parameters $\phi(a)$, $\gamma(a)$, $\beta(a)$, $\rho(a)$, $\mu(a)$; see Example \ref{ex:var_coef}. \label{fig:bifurcation_var_coef}}
\end{figure} 
 
In \Cref{fig:var_coef_reg1}, we see that the solution $i(t,a)$ goes to the infection-free state when $\mathcal{R}_C<1$, as shown in \Cref{th:globalStab}, where global stability was attained. \Cref{fig:var_coef_reg3} illustrates an endemic state when $\mathcal{R}_0>1$; see \Cref{th*}. The case $\mathcal{R}_0<1<\mathcal{R}_C$ is shown in \Cref{fig:var_coef_reg2a,fig:var_coef_reg2b}; by using different initial conditions we have the possibility of multiple steady states as shown in the bifurcation diagram \Cref{fig:bifurcation_var_coef}.

Overall, when parameters are functions of $a$ we can obtain the same behavior as when parameters are constant. This is more realistic but has proven to be challenging to obtain explicit conditions for each case. Numerically, albeit challenging, it is possible to explore the parameter space of the bifurcation~\Cref{fig:bifurcation_var_coef} and produce similar results. Numerical results are highly dependent on the approximation of $B(t)$, which depends on the unknown $i(t,a)$. In order to compute efficiently and accurately $B^j$, we have used Chebyshev expansions based on \texttt{MATLAB}'s library \texttt{Chebfun} \cite{Driscoll2014}. \Cref{fig:bifurcation_var_coef} was computed by solving $\widehat{G}(B) = B$ for several values of $\beta$, where $(s^*(a),i^*(a),r^*(a))$ are computed by using the command \texttt{sum}, which allows to improve running times. Further exploration of the model with variable coefficients is required. 
 
\section{Conclusions} \label{sec:conc}
We have studied an age-structured epidemic model with nonlinear recidivism. The infection-free steady state distribution was computed, as well as its local and global stability. The existence of the endemic non-uniform steady state distribution is guaranteed when $\mathcal{R}_0>1$. Moreover, our analysis shows the existence of multiple endemic equilibria when $\mathcal{R}_0<1$, and necessary conditions that lead to a backward bifurcation were computed. As a numerical example we used alcohol dynamics to showcase the results of the model.

Numerical experiments were conducted to illustrate the different scenarios where there exist multiple drinking states. Typically, $\mathcal{R}_0<1$ is a sufficient condition for a \lq\lq disease" to die out. However, when incorporating social factors, in this case there exists the possibility of an endemic state even when $\mathcal{R}_0<1$. In our model, the {\it basic reproductive number} is a function of the treatment and recovery rate of individuals. 
We have necessary conditions for two positive endemic steady states, which highlights the importance and significance of the initial density of the drinking population. In this case, it is possible to have an endemic state when $\mathcal{R}_0<1$ if the initial number of drinkers is large enough, as seen in \Cref{fig:sol2a1}. Short treatment periods, in turn mostly ineffective, induce more individuals into the temporarily recovered class where the likelihood of relapse is high. This creates a population of new susceptible individuals whose susceptibility and risk can be measured by the strength of environmentally induced recidivism rates. Therefore making prevention strategies far more challenging.

The implications of an age-structure model with nonlinear recidivism could lead to a better understanding of applications where age is an important factor when implementing prevention/control strategies. 


\section*{Acknowledgements}
The authors would like to thank the Research Center in Pure and Applied Mathematics and the Mathematics Department at Universidad de Costa Rica for their support during the preparation of this manuscript. 

\FloatBarrier

\end{document}